\documentclass[10pt]{amsart}

\usepackage{amsmath, amsthm, amscd, amsfonts, amssymb,  graphicx, color}
\usepackage{latexsym}
\usepackage{amscd}
\usepackage{amsthm}
\usepackage{amssymb}
\usepackage{enumerate}
\usepackage{mathabx}

\theoremstyle{plain}
\newtheorem{theorem}{Theorem}
\newtheorem{corollary}[theorem]{Corollary}

\theoremstyle{remark}

\newcommand{\C}{\mathbb{C}}
\newcommand{\R}{\mathbb{R}}

\newcommand*{\Hom}{\ensuremath{\mathrm{Hom\,}}}

\begin{document}

\title{Local polynomials and the Montel theorem}
\author{J.~M.~Almira, L.~Sz\'ekelyhidi}

\let\thefootnote\relax\footnotetext{The research was supported by the Hungarian National Foundation for Scientific Research (OTKA),   Grant No. NK-81402.}

\subjclass[2010]{Primary 43B45, 39A70; Secondary 39B52.}

\keywords{Polynomials on Abelian Groups, Fr\'{e}chet's functional equation, Montel's  Theorem.}

\maketitle
\begin{abstract}
In this paper local polynomials on Abelian groups are characterized by a "local" Fr\'echet--type functional equation. We apply our result to generalize Montel's Theorem and to obtain Montel--type theorems on commutative groups.
\end{abstract}

\section{Introduction}
Polynomials on commutative groups play a basic role in functional equations and in spectral synthesis. The most common definition of polynomial functions depends on Fr\'echet's Functional Equation (see \cite{Fre09, MazOrl34, MR0001560}). Given a commutative group $G$ we denote by $\C G$ the {\it group algebra} of $G$, which is the algebra of all finitely supported complex valued functions defined on $G$. Besides the linear operations (addition and multiplication by scalars) the multiplication is defined by convolution
\begin{equation*}
\mu*\nu(x)=\sum_{y\in G} \mu(x-y)\nu(y)
\end{equation*}
for each $x$ in $G$. With these operations $\C G$ is a commutative complex algebra with identity $\delta_o$, where $o$ is the zero element in $G$, and for each $y$ in $G$ we use the notation $\delta_y$ for the characteristic function of the singleton $\{y\}$. Elements of the form  
\begin{equation*}
\Delta_y=\delta_{-y}-\delta_0
\end{equation*}
of this algebra with $y$ in $G$ are called {\it differences}.
\vskip.3cm

Using the notation $\mathcal C(G)$ for the linear space of all complex valued functions on $G$, it is a module over $\C G$ with the obvious definition
\begin{equation*}
\mu*f(x)=\sum_{y\in G} f(x-y)\mu(y)
\end{equation*}
for each $x$ in $G$. For every function $f$ in the space $\mathcal C(G)$ we shall use the notation $\widehat{f}(x)=f(-x)$, whenever $x$ is in $G$. 
\vskip.3cm

Given a subset $V$ in $\mathcal C(G)$ the {\it annihilator} $V^{\perp}$ of $V$ is the set of all $\mu$'s in $\C G$, for which $\mu*f=0$ for each $f$ in $V$. 
The dual concept is the annihilator $I^{\perp}$ of a subset $I$ in $\C G$: it is the set of all functions $f$ in $\mathcal C(G)$ satisfying $\mu*f=0$ for each $\mu$ in $I$. 

\vskip.3cm

The study of polynomials is related to the study of the annihilators of ideals in $\C G$ generated by products of differences. More exactly, the function $f:G\to\C$ is called a {\it generalized polynomial} of degree at most $n$, if $n$ is a natural number and
\begin{equation}\label{Frech1}
\Delta_{y_1,y_2,\dots,y_{n+1}}*f=0\,,
\end{equation}
where we use the notation $\Delta_{y_1,y_2,\dots,y_{n+1}}$ for the convolution product 
\begin{equation*}
\Delta_{y_1}*\Delta_{y_2}*\dots*\Delta_{y_{n+1}}\,.
\end{equation*}
The smallest $n$ with this property is called the {\it degree} of $f$. In \cite{MR0265798} Djokovi\v c proved that condition \eqref{Frech1}, which is called {\it Fr\'echet's Functional Equation}, is equivalent to the condition
\begin{equation}\label{Frech2}
\Delta_y^{n+1}*f=0\,,
\end{equation}
where $\Delta_y^{n+1}=\Delta_{y_1,y_2,\dots,y_{n+1}}$ with $y=y_1=y_2=\dots=y_{n+1}$. We note that sometimes \eqref{Frech2} is also called Fr\'echet's Functional Equation. 
\vskip.3cm

Polynomials of degree at most one, which vanish at zero, are called {\it additive functions}. They are characterized by the equation
\begin{equation*}
a(x+y)=a(x)+a(y)\,,
\end{equation*}
that is, they are exactly the homomorphisms of $G$ into the additive group of complex numbers. All additive functions on $G$ form a linear space, which is denoted by $\Hom(G,\C)$.
\vskip.3cm

There is a vast literature on different types of polynomials, which play a basic role in the theory of functional equations. In \cite{MR2165414} M.~Laczkovich studies the relations of diverse concepts of polynomials. The reader will find further references and results in this respect in \cite{MazOrl34, MR2582364, MR954205, MR2433311, MR1113488,  MR0001560}.
\vskip.3cm

A special class of generalized polynomials is formed by those functions, which belong to the function algebra generated by the additive functions and the constants. These functions are simply called {\it polynomials}. Hence the general form of a polynomial is
\begin{equation}\label{poly}
p(x)=P\bigl(a_1(x),a_2(x),\dots,a_n(x)\bigr)
\end{equation}
whenever $x$ in $G$, where the functions $a_1,a_2,\dots,a_n:G\to\C$ are additive, and \hbox{$P:\C^n\to\C$} is an ordinary complex polynomial in $n$ variables. In the case $G=\R^n$ or $G=\C^n$ it is well-known (see e.g. \cite{MR1113488}), that every continuous generalized polynomial is a polynomial, in fact, it is an ordinary polynomial. In particular, in this case the additive functions in \eqref{poly} are continuous, assuming that they are linearly independent, which we may always suppose.
In this paper we use the term "ordinary polynomial" for continuous complex valued polynomials on $\R^n$, \hbox{or on $\C^n$.}
\vskip.3cm

The following theorem holds true (see e.g. \cite[Theorem 2. and Theorem 3.]{MR2167990}, \cite[Theorem 4.]{MR2968200}).

\begin{theorem}\label{genpol}
Let $G$ be an Abelian group. Every generalized polynomial on $G$ is a polynomial if and only if the dimension of $\Hom(G,\C)$ is finite.
\end{theorem}

If $G$ is finitely generated, then it is easy to see that every generalized polynomial on $G$ is a polynomial (see e.g. \cite[Theorem 2. and Theorem 3.]{MR2167990}). 
\vskip.3cm

In \cite{Lacz13} M.~Laczkovich introduced the concept of local polynomials. A function $f:G\to\C$ is called a {\it local polynomial}, if its restriction to every finitely generated subgroup is a polynomial. By the previous remark, every generalized polynomial is a local polynomial, however, as it is shown in \cite{Lacz13}, there are local polynomials, which are not generalized polynomials.

\section{A characterization of local polynomials}

In this section we characterize local polynomials by a "local" version of the functional equation \eqref{Frech2}.

\begin{theorem} \label{TLP}
Let $G$ be an Abelian group. The function $f:G\to\C$ is a local polynomial if and only if for each positive integer $t$, and elements $g_1,g_2,\dots,g_t$ in $G$ there are natural numbers $n_i$ such that
\begin{equation}\label{Frech3}
\Delta_{g_i}^{n_i+1}*f(x)=0\,
\end{equation}
holds for $i=1,2,\dots,t$ and for all $x$ in the subgroup generated by the $g_i$'s.
\end{theorem}

\begin{proof}
The necessity of the given condition is obvious. Indeed, if $H$ is the subgroup generated by the $g_i$'s, then the restriction of $f$ to $H$ is a polynomial, hence there is a natural number $n$ such that
\begin{equation*}
\Delta_y^{n+1} f(x)=0
\end{equation*}
holds for each $x,y$ in $H$. Taking $n_i=n$ and $y=g_i$ for $i=1,2,\dots,t$ we get \eqref{Frech3}.
\vskip.3cm

Suppose now that the condition of the theorem is satisfied and let $H$ be the subgroup of $G$ generated by the elements $g_1,g_2,\dots,g_t$, where $t$ is a positive integer. By assumption, there are natural numbers $n_1,n_2,\dots,n_t$ such that 
\begin{equation}\label{loc}
\Delta_{g_i}^{n_i+1}* f(x)=0
\end{equation}
holds for each $x$ in $H$ and for $i=1,2,\dots,t$. Let $N=n_1+n_2+\dots+n_t+t-1$. We show that
\begin{equation}\label{Frech4}
\Delta_y^{N+1}*f(x)=0
\end{equation}
holds for each $x,y$ in $H$.
\vskip.3cm

By \eqref{loc}, we have
\begin{equation}
(\delta_{-g_i}-\delta_0)^{n_i+1}*f=0
\end{equation}
on $H$ for $i=1,2,\dots,t$. Observe that we also have
\begin{equation}\label{locmin}
\Delta_{-g_i}^{n_i+1}* f(x)=0
\end{equation}
for each $x$ in $H$ and $i=1,2,\dots,t$. Indeed, this follows from the obvious identity $$
\delta_{y}-\delta_0=-\delta_y(\delta_{-y}-\delta_0)\,,
$$ 
whenever $y$ is in $G$. Keeping this in mind, in the computation below we shall use the notation $\delta_g^{-m}=\delta_{-g}^m$ for each $g$ in $G$ and positive integer $m$. Let $y$ be in $H$ arbitrary, then there exist nonzero integers $m_1,m_2,\dots,m_t$ such that we have
\begin{equation}\label{fingen}
y=m_1 g_1+m_2 g_2+\dots+m_t g_t\,.
\end{equation}
\vskip.3cm

It follows
\begin{eqnarray*}
&\ & \Delta_y^{N+1}*f =  (\delta_{-y}-\delta_0)^{N+1} *f = (\delta_{-(m_1 g_1+\dots+m_t g_t)}-\delta_0)^{N+1}* f\\
&=& (\delta_{-g_1}^{m_1}\delta_{-g_2}^{m_2}\dots \delta_{-g_t}^{m_t}-\delta_0)^{N+1}* f\\
&=& \bigl[(\delta_{-g_1}^{m_1}\delta_{-g_2}^{m_2}\dots \delta_{-g_t}^{m_t}-\delta_{-g_2}^{m_2}\dots \delta_{-g_t}^{m_t}) 
+ (\delta_{-g_2}^{m_2}\dots \delta_{-g_t}^{m_t}-\delta_{-g_3}^{m_3}\dots \delta_{-g_t}^{m_t}) \\
&\ & \ + (\delta_{-g_3}^{m_3}\delta_{-g_4}^{m_4}\dots \delta_{-g_t}^{m_t}-\delta_{-g_4}^{m_4}\dots \delta_{-g_t}^{m_t})+
(\delta_{-g_4}^{m_4}\dots \delta_{-g_t}^{m_t}-\delta_{-g_5}^{m_5}\dots \delta_{-g_t}^{m_t}) \\
&\ & \  \cdots \\ 
&\ & \ + (\delta_{-g_{t-2}}^{m_{t-2}}\delta_{-g_{t-1}}^{m_{t-1}} \delta_{-g_t}^{m_t}-\delta_{-g_{t-1}}^{m_{t-1}}\delta_{-g_t}^{m_t})+
(\delta_{-g_{t-1}}^{m_{t-1}}\delta_{-g_t}^{m_t}- \delta_{-g_t}^{m_t})+(\delta_{-g_t}^{m_t}-\delta_0)\bigr]^{N+1} *f \\
&=& \bigl[(\delta_{-g_1}^{m_1}-\delta_0)\delta_{-g_2}^{m_2}\dots \delta_{-g_t}^{m_t}+(\delta_{-g_2}^{m_2}-\delta_0)\delta_{-g_3}^{m_3}\dots \delta_{-g_t}^{m_t}\\
&\ & \ \ +\dots+(\delta_{-g_{t-1}}^{m_{t-1}}-\delta_0)\delta_{-g_t}^{m_t}+ (\delta_{-g_t}^{m_t}-\delta_0)\bigr]^{N+1}* f\,.
\end{eqnarray*}

Expanding the $N+1$-th power, by the Multinomial Theorem, we obtain a sum of the form
\begin{equation*}
\sum_{0\leq \alpha_1,\dots,\alpha_t\leq N+1} \frac{(N+1)!}{\alpha_1!\dots \alpha_t!} \prod_{i=1}^t (\delta_{-g_i}^{m_i}-\delta_0)^{\alpha_i} (\delta_{-g_{i+1}}^{m_{i+1}}\dots \delta_{-g_t}^{m_t})^{\alpha_i}f(x)\,,
\end{equation*}
where the sum is also restricted by  $\alpha_1+\alpha_2+\dots+\alpha_t=N+1$, which implies that $\alpha_i\geq n_i+1$ for at least one $1\leq i\leq t$. This implies our statement, as in each term the corresponding $(\delta_{-g_i}^{m_i}-\delta_0)^{\alpha_i}$ factor annihilates $f$, which is clear from
\begin{equation*}
(\delta_{-g_i}^{m_i}-\delta_0)^{\alpha_i}=(\delta_{-g_i}-\delta_0)^{\alpha_i}(\delta_{-g_i}^{m_i-1}+\delta_{-g_i}^{m_i-2}+\dots+\delta_{-g_i}+\delta_0)^{\alpha_i}\,,
\end{equation*}
and the equations \eqref{loc} and \eqref{locmin} (which we use depending on the sign of $m_i$). 
\vskip.3cm

Consequently, equation \eqref{Frech4} holds, which implies that the restriction of $f$ to $H$ is a generalized polynomial. However, on finitely generated Abelian groups every generalized polynomial is a polynomial, hence our theorem is proved.
\end{proof}

We note that the same proof works for a similar statement on commutative semigroups, if the definition of convolution is modified to
\begin{equation*}
f*\mu(x)=\sum_{y\in G} f(x+y) \mu(y)
\end{equation*}
for each $x$ in $G$ with the agreement
\begin{equation*}
\delta_o*f=f
\end{equation*}
for each function $f$.
In that case in \eqref{fingen} the integers $m_i$ are positive.

\section{Connection with Montel--type theorems}
An important contribution to the theory of polynomials is due to P.~Montel. In 1937 in his paper  \cite{Mon37} he proved a surprising result in connection with Fr\'{e}chet's functional equation \eqref{Frech1}. He decided not to focus on the usual regularity approach, that is, assuming some weak smoothness of the generalized polynomial $f$ in order to  conclude that 
$f$ must be an ordinary polynomial. He assumed, instead, that $f$ is a continuous function, and he asked how many steps $h_k$ are necessary to conclude that if 
\begin{equation}\label{Mon1}
\Delta_{h_k}^{n+1}*f(x)=0
\end{equation} 
holds for each $x$ in $\R^d$, then $f$ is an ordinary polynomial. More precisely, he proved the following result.
\begin{theorem}(Montel)
Assume that the additive subgroup of $\mathbb{R}^d$ generated by the vectors $\{h_1,\cdots,h_t\}$ is dense in $\mathbb{R}^d$, further $f:\mathbb{R}^d\to\mathbb{R}$ is continuous and satisfies  \eqref{Mon1} for all $x$ in $\mathbb{R}^d$ and $k=1,\cdots,t$. Then $f$ is an ordinary polynomial. 
\end{theorem}

We remark that the total degree of $f$ may be greater than $n$.
Although Montel's paper appeared in 1937, he had proved the result already in 1935 and, in fact, he gave a talk in Cluj Napoca, Romania, on this subject at that time. The talk was organized by his Ph. D. student,  T.~Popoviciu, who published an improvement of Montel's result in 1936,  prior to its appearance, for the  case $d=1$. In fact, he proved that if $f:\mathbb{R}\to \mathbb{R}$ is a generalized polynomial of degree at most $n$, and  $f$ is continuous at $n+1$ points, then $f$ is an ordinary polynomial of degree at most $n$. Later on  Almira in \cite{Alm14} and Almira and Abu-Helaiel in \cite{AlAb13} applied a completely different approach, using some tools from the theory of translation invariant subspaces, to prove Montel's theorem in several variables not only for continuous functions but also for distributions.
\vskip.3cm

In fact, in the previous section we proved the following Montel--type theorem.

\begin{theorem}\label{Mont1}
Let $G$ be an Abelian group generated by the elements $g_1,g_2,\dots,g_t$. Then $f:G\to\C$ is a polynomial if and only if there are natural numbers $n_1,n_2,\dots,n_t$ such that we have
\begin{equation}\label{Frech5}
\Delta_{g_i}^{n_i+1}*f=0
\end{equation} 
for $i=1,2,\dots,t$.
\end{theorem}

In the subsequent paragraphs we study the relation of Montel--type theorems to local polynomials.
\vskip.3cm

Let $d$ be a positive integer. If $G$ denotes the additive subgroup of $\mathbb{R}^d$ generated by the elements
$\{h_1,\cdots,h_t\}$, then it is well-known \cite[Theorem 3.1]{Wald} that $\overline{G}$, the topological closure of $G$ with the euclidean topology, satisfies  $\overline{G}=V\oplus \Lambda$, where 
$V$ is a vector subspace of $\mathbb{R}^d$ and $\Lambda$ is a discrete additive subgroup of $\mathbb{R}^d$. Furthermore, the case when $G$ is dense in $\mathbb{R}^d$, or, what is the same, the case whenever $V=\mathbb{R}^d$, has been characterized in several different ways (see e.g., \hbox{\cite[Proposition 4.3]{Wald}). } The following theorem is obvious.

\begin{theorem} Assume that $f:\mathbb{R}^d\to\mathbb{R}$ is continuous and its restriction to some dense additive
subgroup of $\mathbb{R}^d$ is a generalized polynomial. Then  $f$ is an ordinary polynomial.
\end{theorem}

\begin{corollary}\label{cor_montel}(Montel's type theorem in several variables) 
Let $t$ be a positive integer, let $n_1,n_2,\dots,n_t$ be natural numbers, further let $f:\mathbb{R}^d\to\mathbb{R}$ be a continuous function satisfying 
\[
\Delta_{h_k}^{n_k+1}f(x)=0
\]
for all  $x$ in $\mathbb{R}^d$ and for $k=1,\cdots,t$. If the subgroup $G$ in $\R^d$ generated by $\{h_1,h_2,\dots,h_t\}$ satisfies $\overline{G}=V\oplus \Lambda$, where $V$ is a vector subspace of $\mathbb{R}^d$, and $\Lambda$ is a discrete additive subgroup of $\mathbb{R}^d$, then there exist ordinary polynomials $p_{\lambda}:\R^d\to\R$ for each $\lambda$ in $\Lambda$
such that 
\[
f(x+\lambda)=p_{\lambda}(x)
\]
holds, whenever $x$ is in $V$ and $\lambda$ is in $\Lambda$. Moreover, we have
\begin{equation*}
\deg p_{\lambda}\leq  n_1+n_2+\dots+n_t+t-1
\end{equation*}
for each $\lambda$ in $\Lambda$. In particular, if $V=\mathbb{R}^d$, then $f$ is an ordinary polynomial. Finally, if $d=1$ and $V=\mathbb{R}$, then $f$ is an ordinary polynomial of degree at most $\min\{n_k:k=1,\cdots,t\}$. 
\end{corollary}

\begin{proof} Let $N=n_1+n_2+\dots+n_t+t-1$. It follows from Theorem \ref{TLP}, when applied to $f_{|G}$, the restriction of $f$ to $G$,  that  
\[
\Delta_h^{N+1}*f(x)=0
\]
for each $x,h$ in $G$. Hence, the continuity of $f$ implies that 
\begin{equation} \label{gen}
\Delta_h^{N+1}*f(x)=0
\end{equation}
for each $x,h$ in $\overline{G}=V\oplus \Lambda$. Consequently,
 $f_{|V}$ is a continuous solution of the functional equation \eqref{Frech2} on $V$. Let $W=V^{\perp}$ denote the orthogonal complement of $V$ in 
$\mathbb{R}^d$ with respect to the standard scalar product. We define the function $F:\R^d\to \R$ by
\[
F(v+w)=f(v)\,,
\]
whenever $v$ is in $V$ and $w$ is in $W$.
Obviously, $F$ is a continuous extension of $f_{|V}$. We claim that $F$ satisfies the  functional equation \eqref{Frech2} on $\R^d$.
Indeed, if we denote by $P_V:\mathbb{R}^d\to\mathbb{R}^d$ the orthogonal projection on $V$, then we have
\begin{eqnarray*}
&\ & \Delta_h^{N+1}*F(x) = \sum_{k=0}^{N+1}\binom{N+1}{k}(-1)^{N+1-k}F(x+kh) \\
&=&  \sum_{k=0}^{N+1}\binom{N+1}{k}(-1)^{N+1-k}F(P_V(x)+kP_V(h)+[(x-P_V(x))+k(h-P_V(h))]) \\ 
& = & \sum_{k=0}^{N+1}\binom{N+1}{k}(-1)^{N+1-k}f(P_V(x)+kP_V(h)) \\
 &=& \Delta_{P_V(h)}^{N+1}*f(P_V(x)) =0. 
\end{eqnarray*}
This implies that $F$ is an ordinary polynomial, whose restriction to $V$ is $f_{|V}$. Thus, if we set $p_0=F$, then we have that $p_0$ is an ordinary polynomial and $f(x)=p_0(x)$ for all $x$ in $V$.
\vskip.3cm

Now let $\lambda$ be arbitrary in $\Lambda$ and we consider the function $g_{\lambda}:V\to\R$ defined by $g_{\lambda}(x)=f(x+\lambda)$ for $x$ in $V$ and $\lambda$ in $\Lambda$. Then equation \eqref{gen} implies that 
$$
\Delta_h^{N+1}*g_{\lambda}(x)=0
$$ 
holds for all $x,h$ in $V$, and the same arguments we used above to define the function $F$ lead to the conclusion that there exists an ordinary polynomial 
$F_{\lambda}:\mathbb{R}^d\to\mathbb{R}$ such that $F_{\lambda}(x)=g_{\lambda}(x)=f(x+\lambda)$ for each $x$ in $V$ and $\lambda$ in $\Lambda$. This proves that  $f(x+\lambda)=p_{\lambda}(x)$ for each $x$ in $V$  with $p_{\lambda}=F_{\lambda}$, which is an ordinary polynomial of degree at most $N$, whenever $\lambda$ is in $\Lambda$.  In particular, if $V=\mathbb{R}^d$ then $f$ is an ordinary polynomial of degree at most $N$. 
\vskip.3cm

Now we assume that $d=1$ and $V=\mathbb{R}$, further let 
$$
m=n_{i_0}=\min\{n_k:k=1,\cdots,t\}\,.
$$ 
Then $f$ is an ordinary polynomial of degree at most $N$, and  $f$ belongs to the annihilator of  $\Delta_h^{m+1}$ with $h=h_{n_{i_0}}\neq 0$. But a simple computation shows that ordinary polynomials which belong to the annihilator of $ \Delta_h^{m+1}$ are ordinary polynomials of degree at most $m$ (see, e.g., \cite[Corollary 1]{Alm14}). The proof is complete. 
\end{proof}

\begin{corollary} \label{local_reg}
Every continuous local polynomial on $\R^d$ is an ordinary polynomial.
\end{corollary}

\begin{proof}
Suppose that the subgroup $G$ of $\R^d$ generated by $h_1,h_2,\dots,h_t$ is dense in $\mathbb{R}^d$. By the definition of local polynomials, $f$ is a polynomial over $G$. This implies that $f$ satisfies the hypotheses of  Corollary \ref{cor_montel} for the group $G$, with $V=\mathbb{R}^d$. Hence $f$ is an ordinary polynomial.    
\end{proof}

\section{The distributional setting}

We recall that if $f$ is a distribution, then its convolution by $\delta_h$ is defined as 
\[
(\delta_h*f)(\phi)=f(\delta_{-h}*\phi),
\] 
where  $\phi$  is an arbitrary test function. This means that for each $h$ in $\R^d$ and for every natural number $m$ we have
\begin{eqnarray*}
(\Delta_h^{m+1}*f)(\phi) &=&  \left(\sum_{k=0}^{m+1}\binom{m+1}{k}(-1)^{m+1-k}\delta_{kh}\right)*f(\phi) \\
&=& \sum_{k=0}^{m+1}\binom{m+1}{k}(-1)^{m+1-k}f(\delta_{-kh}*\phi)\\
&=&  f\left(\sum_{k=0}^{m+1}\binom{m+1}{k}(-1)^{m+1-k}\delta_{-kh}*\phi\right) \\
&=& f(\Delta_{-h}^{m+1}*\phi)\,.
\end{eqnarray*} 
\vskip.3cm

It is reasonable to introduce the following concepts.  Let $f$ be a complex valued distribution on $\mathbb{R}^d$.  We say that $f$  is a {\it generalized polynomial} of degree at most $n$ {\it in  distributional sense}, if  
$$
\Delta_h^{n+1}*f=0
$$
 for all $h$ in $\mathbb{R}^d$.  We say that $f$ a {\it local polynomial in distributional sense}, if for every finitely generated subgroup $H$ of $\mathbb{R}^d$ there exists a natural number $n$ such that $\Delta_h^{n+1}*f=0$ for each $h$ in $H$.

\begin{corollary} \label{cor_dist}
Let $t$ be a positive integer, let $h_1,h_2,\dots,h_t$ be elements in $\R^d$ and let $n_1,n_2,\dots,n_t$ be natural numbers. Suppose that the complex valued distribution $f$ satisfies 
\begin{equation} \label{poli}
\Delta_{h_k}^{n_k+1}*f=0 
\end{equation}
for $k=1,2,\dots,t$. If the vectors $h_1,h_2,\dots,h_t$ generate a dense subgroup in $\mathbb{R}^d$, then $f$ is an ordinary polynomial of degree at most $n_1+n_2+\cdots+n_t+t-1$.
In particular, generalized polynomials and local polynomials in distributional sense are ordinary polynomials.
\end{corollary}

\begin{proof} We let $N=n_1+n_2+\dots+n_t+t-1$.
The very same arguments we applied in Theorem \ref{TLP} show that $\Delta_h^{N+1}*f=0$ holds for every $h$ in the subgroup $G$ generated by the vectors $h_k$, and the density of $G$ in $\mathbb{R}^d$ implies that  $\Delta_h^{N+1}*f=0$ for every $h$ in $\mathbb{R}^d$. A simple computation shows that for each test function $\phi$, for each $t$ in $\mathbb{R}\setminus\{0\}$, and for every any $k\leq d$ we have
\begin{equation*}
0 = \frac{1}{t^{N+1}}\Delta_{te_k}^{N+1}*f(\phi)=
 f\bigl(\frac{1}{t^{N+1}}\Delta_{-te_k}^{N+1}*\phi\bigr)=
 \end{equation*}
 \begin{equation*}
(-1)^{N+1}f\Bigl(\frac{1}{(-t)^{N+1}}\Delta_{-te_k}^{N+1}*\phi\Bigr)  \to (-1)^{N+1}f\Bigl(\frac{\partial^{N+1} \phi}{\partial x_k^{N+1}}\Bigr) = \frac{\partial^{N+1} f}{\partial x_k^{N+1}}(\phi)\,,
\end{equation*}
whenever $t$ tends to $0$.
Assume that $\alpha=(\alpha_1,\cdots,\alpha_d)\in \mathbb{N}^d$ satisfies $|\alpha| =d(N+1)$. Then $\max_{1\leq i\leq d}\alpha_i\geq N$ and we infer
\[
\frac{\partial^{d(N+1)}* f}{\partial x_1^{\alpha_1}\cdots \partial x_d^{\alpha_d}}(\phi)=0
\]
for every test function $\phi$. Hence all partial (generalized) derivatives of $f$ of order $d(N+1)$ are zero, which means that $f$ is an ordinary polynomial. Furthermore, we know that 
$\Delta_h^{N+1}*f=0$ holds whenever $h$ is in $\mathbb{R}^d$. 
Consequently, $f$ is an ordinary polynomial with total degree at most
$N$ (see, for example, \cite[Theorem 3.1]{AlAb13}). 
\end{proof}

We remark that Corollary \ref{cor_dist} applies for functions in $L^p(\mathbb{R}^d)$, since these functions are distributions.


\bigskip

\footnotesize{Jose Maria Almira 

Departamento de Matem\'{a}ticas, Universidad de Ja\'{e}n, Spain

E.P.S. Linares,  C/Alfonso X el Sabio, 28

23700 Linares (Ja\'{e}n) Spain

e-mail address: jmalmira@ujaen.es }

\medskip

\footnotesize{L\'{a}szl{\'o} Sz\'{e}kelyhidi

Department of Mathematics, University of Debrecen, Hungary\\\indent Department of Mathematics, University of Botswana, Botswana

P. O. Box 12, Debrecen 4010, Hungary.

e-mail address: lszekelyhidi@gmail.com}

\end{document}